\documentclass[10pt, a4paper]{article}

\usepackage{amsthm}
\usepackage{amsmath}
\usepackage{amssymb}
\usepackage{graphicx}
\usepackage{amscd}                 
\usepackage[all]{xy}               
\usepackage{stmaryrd}              
\usepackage{wrapfig}
\usepackage{a4wide}
\usepackage{paralist}

\makeatletter
\toks@\expandafter{\@listI}
\edef\@listI{\the\toks@\setlength{\leftmargin}{15pt}}
\makeatother

\usepackage[colorlinks]{hyperref}
\usepackage[color=blue!40, textsize=footnotesize, textwidth=2.5cm]{todonotes}  
\definecolor{dblue}{rgb}{0,0,0.7}
\newtheoremstyle{mythm}{10pt}{10pt}{\color{dblue}}{}{\bf\color{dblue}}{.}{ }{}
\theoremstyle{mythm}
\newtheorem{thm}{Theorem}[section]

\newtheorem{lm}[thm]{Lemma}
\newtheorem{pr}[thm]{Proposition}

\theoremstyle{definition}
\newtheorem{defn}[thm]{Definition}
\newtheorem{exmp}[thm]{Example}

\theoremstyle{remark}
\newtheorem{rmk}[thm]{Remark}

\newcommand\GSp{\mathrm{GSp}}

\newcommand{\Jac}{\mathrm{Jac}}
\newcommand{\End}{\mathrm{End}}

\newcommand{\legendre}[2]{\left(\frac{#1}{#2}\right)}

\newcommand{\Q}{\mathbb Q}
\newcommand{\Z}{\mathbb Z}
\newcommand{\F}{\mathbb F}

\newcommand{\Spec}{\mathrm{Spec}}

\title{Large Galois images for Jacobian varieties of genus $3$ curves}

\author{Sara Arias-de-Reyna, C\'ecile Armana, Valentijn Karemaker,\\ Marusia Rebolledo, Lara Thomas and N\'uria Vila}
\date{}

\begin{document}

\maketitle

\begin{abstract}
Given a prime number $\ell \geq 5$, we construct an infinite family of three-dimensional 
abelian varieties over $\mathbb{Q}$ such that, for any $A/\mathbb{Q}$ in the family, the 
Galois representation $\overline{\rho}_{A,\ell} \colon G_{\mathbb{Q}} \to \mathrm{GSp}_6(\mathbb{F}_{\ell})$ 
attached to the  $\ell$-torsion of $A$ is surjective. Any such variety $A$ will be the 
Jacobian of a genus $3$ curve over $\mathbb{Q}$ whose respective reductions at two auxiliary 
primes we prescribe to provide us with generators of $\mathrm{Sp}_6(\mathbb{F}_{\ell})$.
\end{abstract}



\maketitle

\section*{Introduction}

Let $\ell$ be a prime number. This paper is concerned with realisations of the general symplectic group $\mathrm{GSp}_6(\mathbb{F}_{\ell})$ as a Galois
group over $\mathbb{Q}$, arising from the Galois action on the $\ell$-torsion points of
three-dimensional abelian varieties defined over $\mathbb{Q}$.

More precisely, let $g \geq 1$ be an integer. One can exploit the theory of abelian varieties defined over $\mathbb{Q}$ as follows. 
If $A$ is an abelian variety of dimension $g$ defined over $\Q$, let $A[\ell] = A(\overline{\Q})[\ell]$
denote the $\ell$-torsion subgroup of $\overline{\Q}$-points of $A$. The natural action of the absolute Galois
group $G_{\Q}=\text{Gal}(\overline{\Q}/\Q)$ on $A[\ell]$ gives rise to a continuous Galois representation
$\overline{\rho}_{A,\ell}$ taking values in $\text{GL}(A[\ell]) \simeq \text{GL}_{2g}(\F_{\ell})$.
If the abelian variety $A$ is moreover principally polarised, the image of $\overline{\rho}_{A,\ell}$ lies inside the general
symplectic group $\text{GSp}(A[\ell])$ of $A[\ell]$  with respect to the symplectic pairing induced by the Weil
pairing and the polarisation of~$A$; thus, we have a representation
$$\overline{\rho}_{A,\ell} \: : \: G_{\Q} \longrightarrow \text{GSp}(A[\ell]) \simeq \text{GSp}_{2g}(\F_{\ell}),$$
providing a realisation of
$\text{GSp}_{2g}(\F_{\ell})$ as a Galois group over $\Q$ if  $\overline{\rho}_{A,\ell}$ is surjective.

The image of Galois representations attached to the $\ell$-torsion points of abelian varieties has been widely studied. 
For an abelian variety $A$ defined over a number field, the classical result of Serre ensures surjectivity for almost all primes $\ell$ when  $\mathrm{End}_{\overline{\Q}}(A)=\mathbb{Z}$ and the dimension of $A$ is 2, 6 or odd (cf.~\cite{OeuvresSerre}). More recently, Hall \cite{Hall11} 
proves a result for any dimension, 
with the additional condition that $A$ has semistable reduction of toric dimension 1 at some prime. 
This result has been further generalised to the case of abelian varieties over finitely generated fields 
(cf.~\cite{AGP}).

We can use Galois representations attached to the torsion points
of abelian varieties defined over $\mathbb{Q}$
to address  the Inverse Galois Problem and its
variations involving ramification conditions. For example, the Tame Inverse Galois Problem,
proposed by Birch, asks if, given a finite group $G$,
there exists a tamely ramified Galois extension $K/\mathbb{Q}$ with Galois group isomorphic
to $G$.
Arias-de-Reyna and Vila solved the Tame Inverse Galois problem for $\mathrm{GSp}_{2g}(\mathbb{F}_{\ell})$
when $g=1, 2$ and $\ell
\geq 5$ is any prime number, by constructing a family of genus $g$ curves $C$ such that
the Galois representation $\overline{\rho}_{\mathrm{Jac}(C), \ell}$ attached
to the Jacobian variety $\mathrm{Jac}(C)$ is surjective and tamely ramified for every curve in the family (cf. \cite{SaraNuria09}, \cite{SaraNuria11}).  For both $g=1$ and $g=2$,
the strategy entails determining a set of local conditions at auxiliary primes,
(that is to say, prescribing a finite list of congruences  that the defining equation of $C$ should satisfy)
which ensure the surjectivity of $\overline{\rho}_{\mathrm{Jac}(C), \ell}$, and a careful study of
the ramification at $\ell$ in particularly favourable situations.

In fact, the strategy of ensuring surjectivity of the Galois representation attached to the
$\ell$-torsion of an abelian variety by prescribing local conditions at auxiliary primes
works in great generality.
Given a $g$-dimensional principally polarised abelian variety $A$ over $\mathbb{Q}$,
such that the Galois representation $\overline{\rho}_{A, \ell}$ is surjective,
it is always possible to find some auxiliary primes $p$ and $q$ depending on $\ell$ such that any abelian variety $B$
defined over $\mathbb{Q}$ which is ``close enough'' to $A$ with respect to the primes $p$ and $q$
(in a sense that can be made precise in terms of $p$-adic, resp.~$q$-adic, neighbourhoods  in moduli spaces of principally
polarised $g$-dimensional abelian varieties with full
level structure) also has a surjective $\ell$-torsion
Galois representation $\overline{\rho}_{B,\ell}$.
This is a consequence of Kisin's results on local constancy in $p$-families of Galois
representations; the reader can find a detailed explanation of this aspect in \cite[Section 4.2]{AK13}.

\bigskip

In this paper we focus on the case $g=3$.
Our aim is to find auxiliary primes $p$ and $q$
(depending on $\ell$), and explicit congruence conditions
on polynomials defining genus~$3$ curves, which ensure that any curve $C$, defined by an equation over $\mathbb{Z}$ satisfying
these congruences, will have the property that
the image of $\overline{\rho}_{\mathrm{Jac}(C), \ell}$ coincides with $\mathrm{GSp}_{6}(\mathbb{F}_{\ell})$.
In this way we obtain many distinct realisations of $\mathrm{GSp}_6(\mathbb{F}_{\ell})$ as
a Galois group over $\mathbb{Q}$.

To state our main result, we introduce the following notation: we will say that a polynomial $f(x, y)$ in two variables  is of \emph{3-hyperelliptic
type} if it is of the form $f(x, y)=y^2-g(x)$, where $g(x)$ is a polynomial of degree $7$ or $8$ and of
\emph{quartic type} if the total degree of $f(x, y)$ is $4$.

\begin{thm}\label{thm:main} Let $\ell\geq 13$ be a prime number. For all odd distinct prime numbers $p,q\neq \ell$,  
with $q>1.82\ell^2$, there exist  $f_p(x, y), f_q(x, y)\in\Z[x,y]$ of the same type ($3$-hyperelliptic or quartic), 
such that for any $f(x, y)\in\Z[x,y]$ of the same type as $f_p(x, y)$ and $f_q(x, y)$ and satisfying
\begin{equation*}f(x, y)\equiv f_q(x, y)\pmod{q} \quad \text{ and }\quad  f(x, y)\equiv f_p(x, y)\pmod{p^3},
\end{equation*}
 the image of the Galois representation $\overline{\rho}_{\mathrm{Jac}(C), \ell}$ attached to the $\ell$-torsion points of the Jacobian of the projective genus~$3$ curve $C$ 
 defined over $\Q$ by the equation $f(x,y)=0$ is $\mathrm{GSp}_6(\mathbb{F}_{\ell})$.

Moreover, for $\ell\in\{5,7,11\}$ there exists a prime number $q\neq \ell$ for which the same statement holds for each odd prime number $p\neq q,\ell$.
\end{thm}

In Section \ref{sec:4} we state and prove a refinement of this Theorem (cf.~Theorem \ref{thm:refined}).
In fact, we have a very explicit control of the polynomial $f_p(x, y)$.
In general we can say little about $f_q(x, y)$,
but for any fixed $\ell\geq 13$ and any fixed $q \geq1.82 \ell^2$ we can find suitable polynomials $f_q(x, y)$
by an exhaustive search as follows: there exist only finitely many polynomials $\bar{f}_q(x, y)\in\mathbb{F}_q[x, y]$
of $3$-hyperelliptic or quartic type with non-zero discriminant. For each of these, we can compute the characteristic polynomial of the action of the Frobenius endomorphism on the Jacobian
of the curve defined by $\bar{f}_q(x, y)=0$ by counting the $\mathbb{F}_{q^r}$-points of this curve, for $r=1, 2, 3$, and check whether this polynomial
is an ordinary $q$-Weil polynomial with  non-zero middle coefficient, non-zero trace modulo $\ell$, and which is irreducible modulo $\ell$.
Proposition \ref{irredmodell} ensures that the search will terminate. Then,  any lift of $\bar{f}_q(x, y)$, of the same type, gives us a suitable polynomial $f_q(x, y)\in \mathbb{Z}[x, y]$. In Example \ref{ex:mainthm} we present some concrete examples obtained using \textsc{Sage} and \textsc{Magma}.

Note that the above result constitutes an explicit version of Proposition 4.6 of \cite{AK13} in the case of
principally polarised $3$-dimensional abelian varieties. We can explicitly give the size of the neighbourhoods 
where surjectivity of $\overline{\rho}_{A, \ell}$ is preserved; in other words, we can give the powers of the 
auxiliary primes $p$ and $q$ such that any other curve defined by congruence conditions modulo these powers 
gives rise to a Jacobian variety with surjective $\ell$-torsion representation.

\bigskip

The proof of Theorem \ref{thm:main} is based on two main pillars: the classification of subgroups of $\mathrm{GSp}_{2g}(\mathbb{F}_{\ell})$
containing a non-trivial transvection, and the fact that one can force the image of $\overline{\rho}_{A, \ell}$ to contain a
non-trivial transvection by
imposing  a specific type of ramification
at an auxiliary prime. This strategy goes back to Le Duff \cite{LeDuff98} in the case of Jacobians
of genus $2$ hyperelliptic curves, and has been extended to the general case by Hall in \cite{Hall11}, where he obtains a
surjectivity result for $\overline{\rho}_{A, \ell}$ for almost all primes $\ell$.

We already followed this strategy in \cite{AAKRTV14} to
formulate an explicit surjectivity result for $g$-dimensional
abelian varieties (see Theorem 3.10 of loc.~cit.):
let $A$ be a principally polarised $g$-dimensional abelian variety defined over $\mathbb{Q}$,
such that the reduction of the N\'eron model  of $A$ at some prime $p$ is semistable with toric rank 1, and
the  Frobenius endomorphism at some prime $q$ of good reduction for $A$ acts irreducibly and with trace $a\not=0$ on the
reduction of the N\'eron model of $A$ at $q$. We proved that  for each prime number $\ell\nmid 6pqa$, coprime with the order 
of the component group of 
the N\'eron model of $A$ at $p$, and such that the characteristic polynomial of the  Frobenius endomorphism at $q$ is irreducible mod $\ell$, then the representation $\overline{\rho}_{A, \ell}$ is surjective.

 Section~\ref{sectionone} collects some notations and tools
that we will use in the rest of the paper. In Section~\ref{sec:2}
we address the condition of semistable reduction of toric rank  $1$ at a prime $p$; we obtain a congruence condition modulo $p^3$ (cf.~Proposition~\ref{jacobianthm}).

In Section~\ref{sec:3} we give conditions ensuring that the reduction
of the N\'eron model of a Jacobian variety $A=\mathrm{Jac}(C)$ at a prime $q$ is an absolutely simple abelian variety over $\mathbb{F}_q$ such that the characteristic polynomial
of the Frobenius endomorphism at $q$ is irreducible and has non-zero trace modulo $\ell$ (cf.~Theorem~\ref{A}). We  make use
of Honda-Tate Theory in the ordinary case, which relates so-called ordinary
Weil polynomials to isogeny classes of ordinary abelian varieties  defined
over finite fields of characteristic $q$. First, we need to prove the existence of a
suitable prime $q$ and a suitable ordinary Weil polynomial; this is the content of
Proposition~\ref{irredmodell}, whose proof is postponed to Section~\ref{sectionirred}.
This polynomial provides us  with an abelian variety $A_q$
defined over $\mathbb{F}_q$; any abelian variety $A$
such that the reduction of the N\'eron model of $A$ at $q$ coincides with $A_q$
will satisfy the desired condition at $q$.
At this point we
use the fact that each principally polarised $3$-dimensional abelian variety over $\mathbb{F}_q$ is the Jacobian of
a genus $3$ curve, which can be defined over $\mathbb{F}_q$ up to a quadratic twist.

Once we have established congruence conditions at auxiliary primes $p$ and $q$, we need
to check that any curve $C$ over $\mathbb{Z}$ whose defining equation
satisfies these conditions will provide us with a Galois representation $\overline{\rho}_{\mathrm{Jac}(C), \ell}$
whose image is $\mathrm{GSp}_{6}(\mathbb{F}_{\ell})$. This is carried out in  Section \ref{sec:4}.

\medskip

David Zywina communicated to us that he has recently and independently developed a
method for studying the image of Galois representations $\overline{\rho}_{\mathrm{Jac}(C), \ell}$
attached to the Jacobians of genus $3$ plane quartic curves $C$, for a large class of
such curves (cf.~\cite{Zywina15}). In particular, for each prime $\ell$, he obtains a realisation
of $\GSp_6(\mathbb{F}_{\ell})$ as a Galois group over $\mathbb{Q}$. Samuele Anni, Pedro Lemos and 
Samir Siksek also worked independently on this topic. In their paper \cite{ALS15}, they study
semistable abelian varieties and provide an example of a hyperelliptic genus $3$ curve $C$ such 
that $\mathrm{Im}\overline{\rho}_{\mathrm{Jac}(C), \ell}=\GSp_6(\mathbb{F}_{\ell})$ for all $\ell\geq 3$.
Both Zywina and Anni et al.~propose a method which, given a fixed genus $3$ curve $C$ satisfying suitable
conditions, returns a finite list of primes such that the corresponding representation 
$\overline{\rho}_{\mathrm{Jac}(C), \ell}$
is surjective for any $\ell$ outside the list, generalising the approach of \cite{Dieulefait2002} for the
case of genus $2$ to genus $3$. Both methods rely on Hall's surjectivity result  \cite{Hall11} for the image of Galois representations attached to the torsion points of
abelian varieties as the main technical tool.
In our paper, however, we fix a prime $\ell\geq 5$ and give congruence conditions such that, 
for any genus $3$ curve $C$ satisfying them, we can ensure surjectivity of the attached Galois
representation $\overline{\rho}_{\mathrm{Jac}(C), \ell}$. We also borrow some ideas from Hall's paper \cite{Hall11}, 
although formally we do not make use of his results.

\section{Geometric preliminaries}\label{sectionone}

In this section we recall some background from algebraic geometry and fix some notations.

\subsection{Hyperelliptic curves and curves of genus~3}\label{subsec:notation}

A smooth geometrically connected projective curve\footnote{In this article, we will say that a \emph{curve over a field $K$} 
is an algebraic variety  over $K$ whose irreducible components are of dimension~$1$. (In particular, a curve can be singular.)} $C$ 
of genus $g\geq 1$  over a field $K$ is \emph{hyperelliptic} 
if there exists a degree $2$ finite separable morphism from $C_{\overline K} = C\times_{K} \overline K$ to 
$\mathbb P^1_{\overline K}$.  If~$K$ is algebraically closed or a finite field, then such a curve $C$ has  
a \emph{hyperelliptic equation} defined over $K$\footnote{When $K$ is not algebraically closed nor a finite 
field, the situation can be more complicated (cf.~\cite[Section~4.1]{lercier_ritzenthaler}).}. That is to say, 
the function field of $C$ is $K(x)[y]$ under the relation $y^2+h(x)y=g(x)$ with $g(x), h(x)\in K[x]$, $\deg(g(x))\in \{2g+1, 2g+2\},$ and $\deg(h(x))\leq g$. Moreover, if  $\mathrm{char}(K)\neq 2$, we can take $h(x)=0$. Indeed, in that case,
the conic  defined as the quotient of $C$ by the group generated by the hyperelliptic involution has 
a $K$-rational point, hence is isomorphic to $\mathbb P^1_{K}$ (see e.g.~\cite[Section~1.3]{lercier_ritzenthaler} for more details).
 The curve $C$ is the union of the two affine open schemes
\begin{equation*} \begin{aligned} U &=\Spec \left(K[x,y]/(y^2+h(x)y-g(x))\right)\quad \text{and}\\   V&=\Spec \left(K[t,w]/(w^2+t^{g+1}h(1/t)y-t^{2g+2}g(1/t))\right)\\\end{aligned}\end{equation*}
 glued along $\Spec(K[x,y,1/x]/(y^2+h(x)y-g(x)))$ via the identifications   $x=1/t, y=t^{-g-1}w$.

If $\mathrm{char}(K)\neq 2$, then any separable polynomial $g(x)\in K[x]$ of degree $2g+1$ or $2g+2$ gives rise to a hyperelliptic curve $C$ of genus $g$ defined over $K$ by glueing the  open affine schemes $U$ and $V$ (with $h(x)=0$) as above. We will say
that   $C$ is \emph{given by the hyperelliptic equation $y^2=g(x)$.} We will also say, as in the introduction, that a polynomial in two variables is of \emph{$g$-hyperelliptic type} if it is of the form $y^2-g(x)$ with $g(x)$ a polynomial of degree $2g+1$ or $2g+2$.

\medskip
In this article, we are especially interested in curves of genus~$3$.
If $C$ is a smooth geometrically connected projective non-hyperelliptic curve of genus $3$ defined over a field $K$, 
then its canonical embedding $C\hookrightarrow \mathbb P_{K}^2$
identifies $C$ with a smooth plane quartic curve defined over $K$.
This means that the curve $C$ has a model over $K$ given by $\mathrm{Proj}(K[X,Y,Z]/F(X,Y,Z))$ where $F(X, Y, Z)$ 
is a degree~$4$ homogeneous polynomial with coefficients in $K$. Conversely, any smooth plane quartic curve is the 
image by a canonical embedding of a non-hyperelliptic curve of genus $3$.
If this curve is  $\mathrm{Proj}(K[X,Y,Z]/F(X,Y,Z))$ where $F(X, Y, Z)$ is the homogenisation of a degree~$4$ polynomial 
$f(x, y)\in K[x,y]$,  we will say 
that $C$ is the  \emph{quartic plane curve defined by the affine equation $f(x,y)=0$}.
We will say, as in the introduction,  that a polynomial in two variables is of \emph{quartic type} if its total degree is $4$.

\subsection{Semistable curves and their generalised Jacobians}\label{subsec:jacobians}

We briefly recall the basic notions we need about semistable and stable curves, give the definition of the intersection graph of a curve and explain the link between this  graph  and the structure of their generalised Jacobian.
The classical references we use  are essentially \cite{Liu06} and \cite{BLR90}. For a nice overview which contains other references, the reader could also consult~\cite{romagny}.

A curve $C$ over a field $k$ is said to be  \emph{semistable} if the curve $C_{\overline k}=C\times_{k} \overline k$ is reduced and has at most  ordinary double points as singularities. It is said to be  \emph{stable} if moreover $C_{\overline k}$ is connected, projective of arithmetic genus $\geq 2$, and if any irreducible component of $C_{\overline k}$ isomorphic to $\mathbb{P}^1_{\overline k}$ intersects the other irreducible components in at least three points.  A proper flat morphism of schemes $\mathcal C\to S$  is said to be \emph{semistable} (resp. \emph{stable}) if  it has semistable (resp. stable) geometric fibres.

\medskip
Let  $R$ be a discrete valuation ring with fraction field $K$ and residue field $k$.
Let $C$ be a  smooth  projective  geometrically connected curve over $K$. A \emph{model} of $C$ over $R$ is a normal scheme $\mathcal C/R$ such that $\mathcal C\times_{R} K\cong C$.
We say that $C$ has \emph{semistable reduction} (resp. \emph{stable reduction}) if $C$ has a model $\mathcal C$ over $R$ which is a semistable (resp. stable) scheme over $R$. If such a stable model exists, it is unique up to isomorphism and we call it \emph{the stable model of $C$ over $R$} (cf.~\cite[Chap.10, Definition 3.27 and Theorem 3.34]{Liu06}).
If the curve $C$ has  genus $g\geq 1$, then it admits  a minimal regular model $\mathcal C_{min} $ over $R$, unique up to unique isomorphism. Moreover, $\mathcal C_{min}$ is semistable if and only if $C$ has semistable reduction, and if $g\geq 2$, this is equivalent to $C$ having stable reduction (cf.~\cite[Chap. 10, Theorem 3.34]{Liu06}, or \cite[Theorem 3.1.1]{romagny} when $R$ is strictly henselian).

\medskip
Assume that $C$ is a smooth projective geometrically connected curve of genus $g\geq 2$ over $K$ with semistable reduction. 
Denote by $\mathcal C$ its stable model over $R$ and by $\mathcal C_{min}$ its minimal regular model over $R$.
We know that the Jacobian variety $J=\Jac(C)$ of $C$ admits a N\'{e}ron model $\mathcal J$ over $R$ and the 
canonical morphism $\mathrm{Pic}^0_{\mathcal C/R}\to \mathcal J^0$ is an isomorphism (cf.~\cite[$\S 9.7$, Corollary 2]{BLR90}).
Note that since $\mathcal C_{min}$ is also semistable, we have  $\mathrm{Pic}^0_{\mathcal C_{min}/R}\cong \mathcal J^0$.
Moreover,
the abelian variety $J$ has semistable reduction, that is to say $\mathcal J^0_{k}\cong \mathrm{Pic}^0_{\mathcal C_{ k}/k}$ is canonically an extension of an abelian variety by a torus $T$. As we will see, the structure of the algebraic group $\mathcal J^0_{ k}$ (by which we mean the toric rank and the order of the component group of its geometric special  fibre)  is related to the  intersection graphs of $\mathcal C_{\overline k}$ and $\mathcal C_{min,\overline k}$.

\medskip
Let $X$ be a curve over $\overline{k}$. Consider the \emph{intersection  graph} (or \emph{dual graph}) $\Gamma(X)$, defined as the graph whose vertices are the irreducible components  of $X$, where two irreducible components $X_i$ and $X_j$ are connected by as many  edges as there are irreducible components in the intersection $X_i\cap X_j$.  In particular, if the curve $X$ is semistable, two components $X_i$ and $X_j$ are connected by one edge if there is a singular point lying on both $X_i$ and $X_j$.  Here $X_i=X_j$ is allowed. The \emph{(intersection) graph without loops}, denoted by $\Gamma'(X)$, is the graph obtained by removing from $\Gamma(X)$ the edges corresponding to~$X_i=X_j$.

Next, we paraphrase \cite[$\S 9.2$, Example 8]{BLR90}, which gives the toric rank in terms of the cohomology of the graph $\Gamma(\mathcal C_{\overline k})$.
\begin{pr}[\cite{BLR90}, $\S 9.2$, Ex.~8]\label{BLRExactSeq}
The N\'{e}ron model $\mathcal J$ of the Jacobian of the curve $\mathcal{C}_k$ has semistable reduction.  More precisely, let $X_1,\ldots, X_r$ be the irreducible components of $\mathcal C_{k}$, and let $\widetilde X_1,\dots, \widetilde X_r$ be their respective normalisations. Then the canonical extension associated to $\mathrm{Pic}^0_{\mathcal C_{k}/ k}$ is given by the exact sequence
\[
1\longrightarrow T\hookrightarrow \mathrm{Pic}^0_{\mathcal C_{k}/k}\xrightarrow{\pi^*}\prod_{i=1}^r \mathrm{Pic}^0_{\widetilde X_i/k}
\longrightarrow 1
\]
where the morphism $\pi^*$
is induced by the morphisms $\pi_i:\widetilde X_i\longrightarrow X_i$.
The rank of the torus~$T$ is equal to the rank of the cohomology group $H^1(\Gamma(\mathcal C_{\overline{k}}),\Z). $
\end{pr}
We will use the preceding result in Sections~\ref{sec:2} and \ref{sec:3}. Note that the toric rank does not change if we replace $\mathcal C$ by $\mathcal C_{min}$.

The intersection graph of $\mathcal C_{min,\overline k}$ also determines the order of the component group of the geometric special fibre $\mathcal J_{\overline k}$.
Indeed, the scheme $\mathcal C_{min}\times R^{sh}$, where $R^{sh} $ is the strict henselisation of $R$, fits the hypotheses of \cite[$\S 9.6$, Proposition 10]{BLR90} which gives the order of the component group in terms of the graph of $\mathcal C_{min,\overline k}$; we reproduce it here for the reader's convenience.
\begin{pr}[\cite{BLR90}, $\S 9.6$, Prop.~10]\label{prop:Phi}
Let $X$ be a proper and flat curve over a strictly henselian discrete valuation ring $R$ with algebraically closed residue field $\overline k$. Suppose that $X$  is regular and has geometrically irreducible generic fibre  as well as a geometrically reduced special fibre $X_{\overline k}$. Assume that $X_{\overline k}$ consists of the irreducible components $X_1,\dots, X_r$ and that the local intersection numbers of the $X_i$ are $0$ or $1$ (the latter  is the case if different components intersect at ordinary double points).  Furthermore, assume that the intersection graph without loops $\Gamma'(X_{\overline k})$ consists of $l$ arcs  of edges $\lambda_1,\dots,\lambda_l$,  starting at $X_1$ and ending
at $X_r$, each arc $\lambda_i$ consisting of $m_{i}$ edges. Then the component group $\mathcal J(R^{sh})/\mathcal J^0(R^{sh})$ has order $\sum_{i=1}^l \prod_{j\neq i}m_{j}$.
\end{pr}

We will use this result in the proof of Proposition~\ref{jacobianthm}.

\section{Local conditions at $p$}\label{sec:2}

Let $p>2$ be a prime number. Denote by $\mathbb{Z}_p$ the ring of $p$-adic integers and  by $\mathbb{Q}_p$ the field of $p$-adic numbers.

\begin{defn}\label{polyn}
Let $f(x,y)\in\Z_p[x,y]$ be a polynomial with $f(0,0)=0$ or  $v_p(f(0,0))> 2$. We say that $f(x, y)$ is of type:
\begin{enumerate}
\item[(H)] if
$f(x,y)=y^2-g(x)$, where $g(x)\in\Z_p[x]$ is of degree $7$ or $8$  and such that $$g(x)\equiv x(x-p)m(x)\bmod{p^2\Z_p[x]},$$ with 
$m(x)\in \mathbb{Z}_p[x]$ such that all the roots of its mod $p$ reduction are simple and non-zero;

\item[(Q)] if
$f(x,y)$ is of total degree $4$ and such that  $$f(x,y)\equiv px+x^2-y^2+x^4+y^4 \bmod{p^2\Z_p[x,y]}.$$
\end{enumerate}
\end{defn}

\medskip
For  $f(x,y) \in \Z_p[x,y]$  a polynomial of type (H) or (Q), we will consider  the projective curve $C$ defined by $f(x, y)=0$ as explained in Subsection~\ref{subsec:notation} and 
the scheme $\mathcal C$ over $\Z_p$ defined, for each case of Definition~\ref{polyn} respectively, as follows:
\begin{enumerate}
\item[(H)] the union of the two affine subschemes
$$U=\Spec (\Z_p[x,y]/(y^2- g(x))) \textrm{ and } V=\Spec(\Z_p[t,w]/(w^2-g(1/t)t^8))$$ glued along $\Spec(\Z_p[x,y,1/x]/(y^2-g(x))$ via  $x=1/t, y=t^{-4}w$;
\item[(Q)] the scheme  $\mathrm{Proj}(\Z_p[X,Y,Z]/(F(X, Y, Z)))$, where $F(X, Y, Z)$ is the homogenisation of $f(x, y)$.
\end{enumerate}
This scheme has generic fibre $C$.

\begin{pr}\label{prop:curve}
Let $f(x,y) \in \Z_p[x,y]$ be a polynomial of type (H) or (Q) and $C$ be the projective curve defined by $f(x, y)=0$. 
The  curve $C$ is a smooth projective and geometrically connected curve of genus $3$ over $\Q_p$ with stable reduction. Moreover,  the scheme $\mathcal C$ is the stable model of $C$ over $\Z_p$  and the stable reduction is geometrically integral with exactly one singularity, which is an ordinary double point. 
\end{pr}

\begin{proof}
With the description  we gave in Subsection~\ref{subsec:notation} of what we called the \emph{projective curve defined by $f$}, smoothness  over $\mathbb{Q}_p$ follows from the Jacobian criterion.  This implies that $C$ is a projective curve of genus $3$.

The polynomials defining the affine schemes $U$ and $V$ and the quartic polynomial $F(X, Y, Z)$ are all irreducible over $\overline{\Q}_p$, 
hence over $\Z_p$.  So the curve $C$ is geometrically integral (hence geometrically irreducible and geometrically connected) 
and   $\mathcal C$ is integral as a scheme over $\Z_p$.  It follows in particular that  $\mathcal C$ is flat over $\Z_p$  (cf.~\cite[Chap.~4, Corollary 3.10]{Liu06}). Hence, $\mathcal C$ is a model of $C $ over~$\Z_p$. 

We will show that $\mathcal{C}_{\F_p}$ is semistable (i.e.~reduced with only ordinary double points as singularities) with exactly one singularity. 

Combined with flatness, semistability will imply that the scheme $\mathcal C$ is semistable over $\Z_p$.
Since $C$ has genus greater than $2$, and $C=\mathcal C_{\Q_p}$ is smooth and geometrically connected, 
this is then equivalent to saying that $C$ has stable reduction at $p$ with stable model $\mathcal C$, as required (cf.~\cite[Theorem~3.1.1]{romagny}).

In what follows, 
we denote by $\bar{f}$ the reduction modulo $p$ of any polynomial $f$ with coefficient in~$\Z_p$.  In Case (H),  $\mathcal{C}_{\overline{\F}_p}$   is the union of the two affine subschemes 
$U'=\Spec({\overline{\F}_p}[x,y]/(y^2-x^2\bar m(x)))$ and $V'=\Spec({\overline{\F}_p}[t,w]/(w^2-\bar m(1/t)t^6))$, 
glued along $\Spec(\overline{\F}_p[x,y,1/x]/(y^2-\bar g(x))$ via $x=1/t $ and $y=t^{-4}w$ 
(cf.~\cite[Chap.~10, Example 3.5]{Liu06}).  
In Case (Q), the geometric special fibre is $\mathrm{Proj}({\overline{\F}_p}[X,Y,Z]/(\bar{F}(X, Y, Z)))$. 
In both cases, the defining polynomials are irreducible over ${\overline{\F}_p}$. 
Hence, $\mathcal{C}_{\overline{\F}_p}$ is  integral, i.e.~reduced and irreducible.

Next, we prove that $\mathcal C_{\overline{\F}_p}$  has only one ordinary double point as singularity.
For Case (H), see e.g.~\cite[Chap.~10, Examples 3.4, 3.5 and 3.29]{Liu06}. For Case (Q), we proceed analogously:
first consider the open affine subscheme of $\mathcal{C}_{\overline{\F}_p}$ defined by 
$U=\Spec(\overline{\F}_p[x,y]/\bar{f}(x,y))$, where $\bar{f}(x,y)=x^2-y^2+x^4+y^4\in \mathbb{F}_p[x, y]$. 
Since 
$\mathcal C_{\overline{\F}_p}\backslash U$ is smooth, it suffices to prove that $U$ has only ordinary double singularities.
Let $u\in U$. The Jacobian criterion shows that $U$ is smooth at  $u\neq (0,0)$. 
So suppose that  $u=(0,0)$, and note that $\bar f(x,y)=x^2(1+x^2)-y^2(1-y^2)$.  
Since $2\in\overline\F_p^\times$, there exist  $a(x)=1+xc(x)\in {\overline{\F}_p}[[x]]$ and $b(y)=1+yd(y)\in {\overline{\F}_p}[[y]]$  
such that $1+x^2=a(x)^2 $ and $1-y^2=b(y)^2$, by (\cite[Chap.~1, Exercise 3.9]{Liu06}). Then we have 
$$\widehat{\mathcal O}_{U,u}\cong {\overline{\F}_p}[[x,y]]/(xa(x)+yb(y))(xa(x)-yb(y))\cong {\overline{\F}_p}[[t,w]]/(tw) .$$
It follows that $\mathcal C_{\overline{\F}_p}$ has only one singularity (at $[0:0:1] $) which is an ordinary double singularity.
We have thus showed that $\mathcal C$ is the stable model of $C$ over $\Z_p$ and that its special 
fibre is geometrically integral and has only one ordinary double singularity.

\end{proof}

\begin{pr}\label{jacobianthm} Let $f(x,y) \in \Z_p[x,y]$ be a polynomial of type (H) or (Q) and $C$ be the projective curve defined by $f(x, y)=0$. 
 The Jacobian variety $\mathrm{Jac}(C)$ of the curve $C$  has a N\'eron model $\mathcal J$ over 
 $\Z_p$ which has semi-abelian reduction  of  toric rank  $1$. The component group of the geometric 
 special fibre of $\mathcal J$ over $\overline \F_p$ has order $2$.

\end{pr}

\begin{proof}
By Proposition~\ref{prop:curve}, the curve $C$ is a smooth projective geometrically connected curve of genus $3$ over $\Q_p$ with stable reduction and stable model $\mathcal C$ over $\Z_p$. 
Let   $\mathcal{C}_{min}$ 
be the minimal regular model of $C$. As recalled in Subsection~\ref{subsec:jacobians},  $\mathrm{Jac}(C)$ admits a N\'eron model 
 $\mathcal J$  over $\Z_p$ and  the canonical morphism 
$\mathrm{Pic}^0_{\mathcal C/\Z_p}\to \mathcal J^0$ is an isomorphism. In particular, $\mathcal J$ 
has semi-abelian reduction and $\mathcal J^0_{\F_p}\cong\mathrm{Pic}^0_{\mathcal C_{\F_p}/{\F_p}}$. 
Since $\mathcal C_{min}$ is also semistable, 
we have $\mathrm{Pic}^0_{\mathcal C_{min}/S}\cong \mathcal J^0$. 

By Proposition \ref{BLRExactSeq}, the  toric rank of $\mathcal J^0_{\overline\F_p}$ is equal to the 
rank of the cohomology group of the dual graph of $\mathcal C_{\overline \F_p}$.  
Since $\mathcal C_{\overline \F_p}$ is  irreducible and has only one ordinary double point, 
the dual graph consists of one vertex and one loop, so the rank of $\mathcal J^0_{\overline \F_p} $ is $1$.

To determine the order of the component group  of the geometric special fibre $\mathcal J_{\overline \F_p}$, 
we apply Proposition \ref{prop:Phi} to the minimal regular model $\mathcal C_{min}\times \Z_p^{sh}$, 
where $\Z_p^{sh} $ is the strict henselisation of $\Z_p$. This is still regular and semistable 
over $\Z_p^{sh}$ (cf.~\cite[Chap.~10, Proposition~3.15-(a)]{Liu06}).
Let $e$ denote the thickness of the ordinary double point  of $\mathcal C_{\overline\F_p}$ 
(as defined in \cite[Chap.~10, Definition3.23]{Liu06}). Then by \cite[Chap.~10, Corollary 3.25]{Liu06}, 
the geometric special fibre $\mathcal C_{min,\overline \F_p} $ of $\mathcal C_{min}\times \Z_p^{sh}$ 
consists of a chain of $e-1$ projective lines over $\F_p$ and one component of genus $2$ 
(where the latter corresponds to the irreducible component $\mathcal C_{\overline \F_p}$), which meet transversally at rational points.
It follows from Proposition \ref{BLRExactSeq}
that  the order of the component group $\mathcal J(\Z_p^{sh})/\mathcal J^0(\Z_p^{sh})$ of the geometric special fibre  is equal to the thickness~$e$.

We will now show that in both cases (H) and (Q), the thickness $e$ is equal to $2$, which will 
conclude the proof of Proposition~\ref{jacobianthm}. For this, in several places, we will use 
the well-known  fact that every formal power series in $\Z_p[[x]]$ (resp. $\Z_p[[y]]$, $\Z_p[[x,y]]$) 
with constant term $1$ (or more generally a unit square in $\Z_p$) is a square in $\Z_p[[x]]$ 
(resp. $\Z_p[[y]]$, $\Z_p[[x,y]]$) of some invertible formal power series. 

Let $U$ denote the affine subscheme $\Spec(\Z_p[x,y]/(f(x,y)))$ which contains the ordinary double point $P=[0:0:1]$. 
Firstly, we claim that, possibly after a finite extension of scalars $R/\Z_p$ which splits the singularity, in both cases we may write in $R[[x,y]]$:
\begin{equation}\label{formf}\pm f(x,y)=x^2a(x)^2-y^2b(y)^2+p\alpha x+p^2yg(x,y)+p^{r}\beta
\end{equation}
where $ a(x) \in R[[x]]^\times,b(y) \in R[[y]]^\times, g(x,y)\in\Z_p[x,y], \alpha\in\Z_p^\times$, $\beta\in\Z_p$. Moreover, from the assumptions on $f$, it follows that either $\beta=0$, or $\beta \in \Z_p^\times $ and $r=v_p(f(0,0))> 2$.  

We prove the claim case by case:
\begin{enumerate}
\item[(H)]  We have $f(x,y)=y^2-g(x)=y^2-x(x-p)m(x)+p^2h(x)$ for some $h(x)\in\Z_p[x]$.  Since  $h(x)=h(0)+xs(x)$ for some 
$s(x)\in\Z_p[x]$  and $m(x)+ps(x)=m(0)+p s(0)+xt(x)$  for some $t(x)\in\Z_p[x]$, we obtain
\begin{eqnarray*}
f(x,y) &=& y^2-x^2m(x)+px(m(x)+ps(x))+p^2h(0)\\
        &=& y^2-x^2(m(x)-pt(x))+px(m(0)+ps(0))+p^2h(0).
\end{eqnarray*}

Since $m(0)\neq 0\pmod p,$ we have $m(0)-pt(0)\in\Z_p^\times, $ hence if we extend  the scalars to some finite extension  
$R$ over $\Z_p$, in which $m(0)-pt(0)$ is a square, we get that $(m(x)-pt(x))$ is a square of some $a(x)$ in $R[[x]]^\times$. 
Then $-f(x,y)$ has the expected form.
 Note that $R/\Z_p$ is unramified because $p\neq 2$ and $m(0)\neq 0\pmod p$, so we still denote  the ideal of $R$ above $p\in\Z_p$ by $p$.

\item[(Q)] We have $f(x,y)=x^4+y^4+x^2-y^2+px+p^2h(x,y)$ for some $h(x,y)\in\Z_p[x,y]$.
We may write $h(x,y)=\delta+x\gamma+x^2s(x)+yt(x,y)$ for some $\gamma,\delta\in\Z_p$, $s(x) \in\Z_p[x]$ and $t(x,y)\in\Z_p[x,y]$.  We obtain  
\begin{equation*}\begin{aligned}
f(x,y)&=x^2(1+x^2)-y^2(1-y^2)+px+p^2(\delta+x\gamma+x^2s(x)+yt(x,y))\\
 & =  x^2(1+x^2+p^2 s(x))-y^2(1-y^2)+px(1+p\gamma)+p^2yt(x,y)+p^2\delta.
\end{aligned}\end{equation*}
Since $1+x^2+p^2 s(x)$ and $1-y^2$ have constant terms which are squares in $\Z_p^\times$, the formal power series are 
squares in $\Z_p[[x]]$, resp. $\Z_p[[y]]$.   So $f(x,y)$ again has the desired form.

\end{enumerate}
Next, we show that  $e=2$ for $\pm f(x,y)$ of the form \eqref{formf}. In $R[[x,y]]$,  
we have
$$\pm f(x,y)=\left(xa(x)+p\frac{\alpha}{2a(x)}\right)^2-\left(yb(y)-p^2\frac{g(x,y)}{2b(y)}\right)^2+p^2c(x,y),$$
where $c(x,y)=p^{r-2}\beta-\frac{\alpha^2}{4a(x)^2}+p^2\frac{g(x,y)^2}{4b(y)^2}$. Since either $\beta=0$ or  $r>2$ and $\frac{\alpha^2}{4a(0)^2}\not \equiv 0\pmod p$,  
the constant term $\gamma$ of the formal power series $c(x, y)$ belongs to $R^\times$. It follows that  $\gamma^{-1}c(x,y)$ is the square of some formal power series 
$d(x,y)\in R[[x,y]]^\times$. Defining the variables
$$u=\frac{xa(x)}{d(x,y)}+p\frac{\alpha}{2a(x)d(x,y)}-\frac{yb(y)}{d(x,y)}+p^2\frac{g(x,y)}{2b(y)d(x,y)}$$ and 
$$v= \frac{xa(x)}{d(x,y)}+p\frac{\alpha}{2a(x)d(x,y)}+\frac{yb(y)}{d(x,y)}-p^2\frac{g(x,y)}{2b(y)d(x,y)},$$ we get
$ \widehat O_{U\times R,P}\cong R[[u,v]]/(uv\pm p^2\gamma)$. Since  $ \gamma\in R^\times$,  it follows that  $e=2$.

\end{proof}

\section{Local conditions at $q$}\label{sec:3}

This section is devoted to the proof of the following key result.
In the statement, the two conditions on the characteristic polynomial, namely non-zero trace and irreducibility modulo $\ell$, are the ones appearing in Theorem~2.10 of~\cite{AAKRTV14} which is used to prove the main Theorem~\ref{thm:main}.

\begin{thm}\label{A} Let $\ell \geq 13$ be a prime number. For every prime number $q>1.82\ell^2$, there exists a smooth geometrically connected curve $C_q$ of genus $3$ over $\F_q$ whose Jacobian variety $\Jac(C_q)$ is a $3$-dimensional ordinary absolutely simple abelian variety such that the characteristic polynomial of its Frobenius endomorphism is irreducible modulo $\ell$ and has non-zero trace modulo $\ell$.

Moreover, for  $\ell\in\{3,5,7,11\}$, there exists a prime number $q>1.82\ell^2$ such that the same statement holds.
\end{thm}

For any integer $g \geq 1$, a $g$-dimensional abelian variety over a finite field $k$ with $q$ elements is said to be \emph{ordinary} if its group of  $\mathrm{char}(k)$-torsion points has rank $g$.

The proof of Theorem~\ref{A} relies on Honda-Tate theory, which relates  abelian varieties to  Weil polynomials:

\begin{defn}\label{weilpol}
A \emph{Weil $q$-polynomial}, or simply a \emph{Weil polynomial},  is a monic polynomial $P_q(X) \in \Z[X]$ of even degree $2g$ whose complex roots are all \emph{Weil $q$-numbers}, i.e., algebraic integers with absolute value $\sqrt{q}$ under all of their complex embeddings.
Moreover, a Weil $q$-polynomial is said to be \emph{ordinary} if its middle coefficient is coprime to $q$. 
\end{defn}
In particular, for $g=3$, every Weil $q$-polynomial of degree~6 is of the form
\begin{equation*}
P_q(X)=X^6+aX^5+bX^4+cX^3+qbX^2+q^2aX+q^3
\end{equation*}
for some integers $a$, $b$ and $c$ (cf.~\cite[Proposition 3.4]{Howe95}). Such a Weil polynomial is ordinary if, moreover, $c$ is coprime to $q$.

Conversely, not every polynomial of this form is a Weil polynomial. However, we will prove in  Proposition~\ref{boundqweil}  that for $q > 1.82 \ell^2$, 
every polynomial as above with  $|a|,|b|,|c| <\ell$ is a Weil $q$-polynomial.

As an important example, the characteristic polynomial of the Frobenius endomorphism of an 
abelian variety over $\mathbb{F}_q$ 
is a Weil $q$-polynomial, by the Riemann hypothesis as proven by Deligne.\\

A variant of the Honda-Tate Theorem (cf.~\cite[Theorem~3.3]{Howe95}) states that the map which sends an ordinary abelian variety over $\F_q$ to the characteristic polynomial of its Frobenius endomorphism induces a bijection between the set of isogeny classes of ordinary abelian varieties of dimension $g \geq 1$ over $\F_q$ and the set of ordinary Weil $q$-polynomials of degree $2g$. Moreover, under this bijection, isogeny classes of simple ordinary abelian varieties correspond to irreducible ordinary Weil $q$-polynomials.\\

Hence, the proof of Theorem~\ref{A} consists in proving the existence of an irreducible ordinary Weil $q$-polynomial 
of degree 6 which gives rise to an isogeny class of simple ordinary abelian varieties of dimension 3. 
By Howe (cf.~\cite[Theorem~1.2]{Howe95}), such an isogeny class contains a principally polarised abelian variety $A$ 
over $\F_q$, which is the Jacobian variety of some curve $C_q$ defined over $\overline{\mathbb F}_q$ by results due to Oort and Ueno. 
If this abelian variety $A$ is moreover absolutely simple, the curve is geometrically irreducible and we can conclude by a Galois descent argument. 
Thus, it is a natural question whether the Weil $q$-polynomial determines if the abelian varieties in the isogeny class are absolutely simple.

In \cite{HoweZhu02}, Howe and Zhu give a sufficient condition for an abelian variety over a finite field to be absolutely simple; for ordinary varieties, this condition is also necessary. Let $A$ be a simple abelian variety over a finite field, $\pi$ its Frobenius endomorphism and $m_A(X)\in\mathbb{Z}[X]$ the minimal polynomial of $\pi$. Since $A$ is simple, the subalgebra $\Q(\pi)$ of $\End(A)\otimes \Q$  is a field; it contains a filtration of  subfields  $\Q(\pi^d)$ for  $d>1$. If moreover $A$ is ordinary, then the fields $\End(A)\otimes \Q=\Q(\pi) $ and $\Q(\pi^d)$ $(d>1)$ are all CM-fields, i.e., totally imaginary quadratic extensions of a totally real field. A slight reformulation of Howe and Zhu's criterion is the following (see Proposition~3 and Lemma~5 of \cite{HoweZhu02}):
\begin{pr}[Howe-Zhu criterion for absolute simplicity]\label{howe-zhu} Let $A$ be a simple abelian variety over a finite field $k$.
If $\Q(\pi^d)=\Q(\pi)$ for all integers $d>0$, then $A$ is absolutely simple. If $A$ is ordinary, then the converse is also true, and if $\Q(\pi^d)\neq\Q(\pi)$ for some $d>0$, then $A$ splits over the degree $d$ extension of $k$.
Moreover, if $\Q(\pi^d)$ is a proper subfield of $\Q(\pi)$ such that $\Q(\pi^r)=\Q(\pi)$ for all $r<d$, then either $m_A(X)\in\Z[X^d] $, or  $\Q(\pi)=\Q(\pi^d,\zeta_d)$ for a primitive $d$-th root of unity $\zeta_d$.
\end{pr}

From this criterion, Howe and Zhu give elementary conditions for a simple $2$-dimensional abelian variety to be absolutely simple, see~\cite[Theorem~6]{HoweZhu02}.  Elaborating on their criterion and inspired by \cite[Theorem~6]{HoweZhu02}, we prove the following  for dimension 3:

\begin{pr}\label{abssimple}
Let $A$ be an ordinary simple abelian variety of dimension $3$ over a finite field $k$ of odd cardinality $q$. Then either $A$ is absolutely simple or  the characteristic polynomial of the  Frobenius endomorphism of $A$ is of the form  $X^6+cX^3+q^3$ with $c$ coprime to $q$  and $A$ splits over the degree~$3$ extension of $k$.
\end{pr}

\begin{proof}
Let $A$ be an ordinary  simple but not absolutely simple abelian variety of dimension $3$ over~$k$.  
Since $A$ is simple, the characteristic polynomial of $\pi$ is $m_A(X)$. We apply Proposition~\ref{howe-zhu}  to $A$:  
Let $d$ be the smallest integer such that $\Q(\pi^d)\neq\Q(\pi)$.
Either $m_A(X)\in\Z[X^d]$ or there exists a $d$-th root of unity $\zeta_d$ such that $\Q(\pi)=\Q(\pi^d,\zeta_d)$.

We will prove by contradiction that  $m_A(X)\in\Z[X^d]$.  Since  $m_A(X)$ is ordinary, the coefficient of degree $3$ is non-zero, and it will follow that  $d=3$ and  that $m_A(X)$ has the form $X^6+cX^3+q^3$, proving the proposition.

So, suppose that  $m_A(X)\not\in\Z[X^d]$.
The field $K=\Q(\pi)=\Q(\pi^d,\zeta_d)$ is a CM-field of degree $6$ over $\Q$, hence
its proper CM-subfield $L=\Q(\pi^d)$  has  to be a quadratic imaginary field.
It follows that $\phi(d)=3$ or $6$, where $\phi$ denotes the Euler totient function.
 However, $\phi(d)=3$ has no solution, so we must have $\phi(d)=6$, i.e. $d\in\{7,9,14,18\}$,  and $K=\Q(\zeta_d)$. 
Note that $\Q(\zeta_7)=\Q(\zeta_{14})$ and $\Q(\zeta_9)=\Q(\zeta_{18})$, and  they contain only one quadratic 
imaginary field; namely, $\Q(\sqrt{-7})$ for $d=7 $ (resp. $14$), and $\Q(\sqrt{-3})$ for $d=9$ (resp. $d=18$) 
(cf.~\cite{washington}).
Let $\sigma$ be a generator of the (cyclic) group $\mathrm{Gal}(K/L)$ of order~$3$. In their proof of \cite[Lemma 5]{HoweZhu02}, 
Howe and Zhu show that we can choose $\zeta_d$ such that $\pi^\sigma=\zeta_d \pi$.
Moreover, $\zeta_d^\sigma=\zeta_d^k$ for some integer $k$ (which can be chosen to lie in $[0,d-1]$). Since $\sigma$ is of order $3$, we have $\pi=\pi^{\sigma^3}=\zeta_d^{(k^2+k+1)}\pi$, which gives $k^2+k+1 \equiv 0\pmod d$.  This rules out the case $d=9$ and $18$, because $-3$ is neither a square modulo $9$ nor a square modulo $18$. So $d=7$ or~$14$,  $K=\Q(\zeta_7)$  and $\Q(\pi^d)=\Q(\sqrt{-7})$.
It follows that the characteristic polynomial of $\pi^d$, which   is of the form
 \[
 X^6+\alpha X^5+\beta X^4+\gamma X^3+\beta q^d X^2+\alpha q^{2d} X+q^{3d}\in\Z[X],
\]
  is the cube of a quadratic polynomial of discriminant $-7$. This is true  if and only if
 \begin{equation*}\label{quadequ}
\alpha^2-36q^d+63=0,\quad\alpha^2-3\beta+9q^d=0\quad\mbox{and}\quad  \alpha^3-27\gamma+54\alpha q^d=0,
 \end{equation*}
 that is,
 \begin{equation*}\label{simplequadequ}
 \alpha^2=9(4q^d-7),\quad \beta=3(5q^d-7)\quad\mbox{and}\quad  3\gamma=\alpha(10q^d-7).
 \end{equation*}
However, the first equation has no solution in $q$. Indeed, suppose that $4q^d-7$ is a square, 
say $u^2$ for some integer $u$. Then $u$ is odd, say $u=1+2t$ for some integer $t$, 
hence $4q^d=8+4t(t+1)$, so $2$ divides $q$, which contradicts the hypothesis.

Hence, we obtain that  $m_A(X)\in\Z[X^d]$ and Proposition~\ref{abssimple} follows.
\end{proof}

Finally, the proof of Theorem~\ref{A} relies  on  Proposition~\ref{abssimple} and the following proposition, 
whose proof consists on counting arguments and is postponed to Section~\ref{sectionirred}:

\begin{pr}\label{irredmodell}
 For any prime number  $\ell\geq 13$ and any prime number $q>1.82\ell^2$, 
 there exists an ordinary Weil $q$-polynomial   $P_q(X)=X^6+aX^5+bX^4+cX^3+qbX^2+q^2aX+q^3$, 
 with $a\not\equiv 0\pmod \ell$, which is irreducible modulo $\ell$. For $\ell\in\{3,5,7,11\}$, 
 there exists some prime number $q>1.82\ell^2$  and an ordinary Weil $q$-polynomial as above. 
 Moreover, for all $\ell \geq 3$, the coefficients $a,b,c$ can be chosen to lie in $\Z\cap [-(\ell-1)/2,(\ell-1)/2]$.
\end{pr}

\begin{rmk}\label{remirred} Computations suggest that for $\ell\in\{5,7,11\}$ and 
\emph{any} prime number $q > 1.82\ell^2$, there still exist integers $a,b,c$ such 
that Proposition~\ref{irredmodell} holds. For $\ell=3$, this is no longer true: our 
computations indicate that if $q$ is such that $\legendre q\ell=-1$, then there are no 
suitable $a,b,c$, while if $q$ is such that $\legendre q\ell=1$, they indicate that 
there are $4$ suitable triples $(a,b,c)$. \end{rmk}

We now have all the ingredients to prove Theorem~\ref{A}.

\begin{proof}[Proof of Theorem~\ref{A}.]
Let $\ell$ and $q$ be two distinct prime numbers as in Proposition~\ref{irredmodell} 
and let $P_q(X)$ be an ordinary Weil $q$-polynomial provided by this proposition.  
Since the polynomial $P_q(X)$ is irreducible modulo $\ell$, it is a fortiori irreducible over $\Z$. 
It is also ordinary and of degree $6$. Hence, by Honda-Tate theory, it defines an isogeny class $\mathcal A$ of ordinary simple abelian varieties of dimension $3$ over $\F_q$. By Proposition~\ref{abssimple}, since $a\neq 0$, the abelian varieties in $\mathcal A$ are actually absolutely simple. Moreover, according to Howe (cf.~\cite[Theorem~1.2]{Howe95}), $\mathcal A$ contains a principally polarised abelian variety $(A,\lambda)$.

Now, by the results of Oort-Ueno (cf.~\cite[Theorem~4]{OortUeno73}), there exists a so-called good curve $C$ defined over $\overline{\F}_q$ such that $(A,\lambda)$ is $\overline{\F}_q$-isomorphic to $(\mathrm{Jac}(C),\mu_0)$, where $\mu_0$ denotes the canonical polarisation on $\mathrm{Jac}(C)$.
A curve over $\overline{\F}_q$ is a \emph{good curve} if it is either irreducible and non-singular or a non-irreducible stable curve whose generalised Jacobian variety is an abelian variety (cf.~\cite[Definition (13.1)]{Howe95}). In particular, the curve $C$ is stable, and so semi-stable. Since the generalised Jacobian variety $\mathrm{Jac}(C)\cong\mathrm{Pic}^0_{C}$ is an abelian variety, the torus appearing in the short exact sequence of Proposition~\ref{BLRExactSeq} is trivial. Hence, there is an isomorphism $\mathrm{Jac}(C) \cong \prod_{i=1}^r \mathrm{Pic}^0_{\widetilde{X_i}}$, where $\widetilde{X_1},\ldots,\widetilde{X_r}$ denote the normalisations of the irreducible component of $C$ over $\overline{\F}_q$. Since $\mathrm{Jac}(C)$ is absolutely simple, we conclude that $r=1$, i.e., the curve $C$ is irreducible, hence smooth.

We can therefore apply Theorem~9 of the appendix by Serre in~\cite{Lauter01} (see also the reformulation in \cite[Theorem~1.1]{Ritzenthaler10}) and conclude that the curve $C$ descends to $\F_q$. Indeed, there exists a smooth and geometrically irreducible curve $C_q$ defined over $\F_q$ which is isomorphic to $C$ over $\overline{\F}_q$. Moreover, either $(A,\lambda)$ or a quadratic twist of $(A,\lambda)$ is isomorphic to $(\mathrm{Jac}(C_q), \mu)$ over $\F_q$, where $\mu$ denotes the canonical polarisation of $\mathrm{Jac}(C_q)$. The characteristic polynomial of $\mathrm{Jac}(C_q)$ is  $P_q(X)$ or $P_q(-X)$, since the twist may replace the Frobenius endomorphism with its negative.

Note that the polynomial $P_q(-X)$ is still an ordinary Weil polynomial which is irreducible modulo $\ell$ with non-zero trace, and $\Jac(C_q)$ is still ordinary and absolutely simple. This proves Theorem~\ref{A}.
\end{proof}

\begin{rmk}\label{sek} In the descent argument above, the existence of a non-trivial 
quadratic twist may occur in the non-hyperelliptic case only. This obstruction for an 
abelian variety over $\overline{\mathbb{F}}_q$ to be a Jacobian over $\F_q$ was first stated by Serre 
in a Harvard course~\cite{Serre85}; it was derived from a precise reformulation of Torelli's 
theorem that Serre attributes to Weil~\cite{Weil57}.
Note that Sekiguchi investigated the descent of the curve in~\cite{Sekiguchi81}  and~\cite{Sekiguchi86}, 
but, as Serre pointed out to us, the non-hyperelliptic case was incorrect. According to MathSciNet review MR1002618 
(90d:14032), together with Sekino, Sekiguchi corrected this error in~\cite{SekiSeki88}.
\end{rmk}

\section{Proof of the main theorem}\label{sec:4}

The goal of this section is to prove Theorem \ref{thm:main}, by collecting together the results
from Sections \ref{sec:2} and \ref{sec:3}. We keep the notation introduced in Subsection \ref{subsec:notation};
in particular, we will consider genus $3$ curves defined by polynomials which are of $3$-hyperelliptic or quartic type.
We will prove the following refinement of Theorem \ref{thm:main}:

\begin{thm}\label{thm:refined}
Let $\ell\geq 13$ be a prime number. For each prime number $q>1.82 \ell^2$, there exists $\bar{f}_q(x, y)\in \mathbb{F}_q[x, y]$ of $3$-hyperelliptic or
 quartic type, such that if $f(x, y)\in \mathbb{Z}[x, y]$ is a lift of  $\bar{f}_q(x, y)$, of the same type, satisfying the following two conditions for some prime number $p\not\in\{2, q, \ell\}$:
\begin{enumerate}
 \item $f(0, 0)=0$ or $v_p(f(0, 0))>2$;
 \item $f(x,y)$ is congruent modulo $p^2$ to:
$$\begin{cases} y^2 - x(x - p)m(x) &\text{ if } \bar{f}_q(x, y) \text{ is of hyperelliptic type}\\

                         x^4 + y^4 + x^2 -y^2 + px  &\text{ if } \bar{f}_q(x, y) \text{ is of quartic type}\\
                         \end{cases}$$
for some $m(x) \in \mathbb{Z}_p[x]$ of degree 5 or 6 with simple non-zero roots modulo $p$;

\end{enumerate}
\noindent then the projective  curve $C$ defined over $\mathbb{Q}$ by the equation $f(x, y)=0$ is a smooth projective geometrically irreducible
genus $3$ curve, such that the image of the Galois representation $\overline{\rho}_{\mathrm{Jac}(C), \ell}$ attached to the $\ell$-torsion of $\mathrm{Jac}(C)$
coincides with $\mathrm{GSp}_6(\mathbb{F}_{\ell})$.

Moreover, if $\ell\in \{5, 7, 11\}$, the statement is true, replacing ``For each prime number $q$'' by ``There exists an odd prime number $q$''.
\end{thm}

\begin {rmk}\label{rem:CRT}
 Let $\ell\geq 5$ be a prime number.  Note that it is easy to construct infinitely many  polynomials
 $f(x,y)$ satisfying the conclusion of Theorem~\ref{thm:refined}:
 choose a polynomial $f_p(x, y)$ satisfying the conditions in Definition~\ref{polyn}. Choose a prime $q>1.82 \ell^2$, 
 and find a polynomial $\bar{f}_q(x,y)$ that satisfies the conditions in Proposition \ref{irredmodell} (e.g.~by
a computer search based on the method suggested after Theorem 0.1). Then
 it suffices to choose each coefficient of $f(x, y)$ as a lift of the corresponding coefficient of $\bar{f}_q(x,y)$
 to an element of $\mathbb{Z}$,
 which is congruent mod $p^3$ to the corresponding coefficient of $f_p(x, y)$.
 This also proves that Theorem~\ref{thm:main} follows from Theorem~\ref{thm:refined}.
\end {rmk}

\begin{exmp}\label{ex:mainthm}
\begin{enumerate}
\item For $\ell=13$, we choose $p=7$, $q=313$. A computer search produces the polynomial 
$\bar{f}_q(x,y)=y^2-(x^7+x-1)$, which defines a hyperelliptic genus $3$ curve over $\mathbb{F}_q$. 
Let $f_p(x,y)=y^2-x(x-7)(x-1)(x-2)(x-3)(x-4)(x-5)$.
Using the Chinese Remainder Theorem we construct the hyperelliptic curve over $\mathbb{Q}$ with equation $f(x,y)=0$, where
\begin{multline*}
f(x,y)=y^2 -( x^7-14085 x^6 + 33804x^5 -27231 x^4 \\ + 27231x^3 -35995 x^2 -33803x + 25039).
\end{multline*}
\item For $\ell=5$, we choose $p=3$, $q=97$. Through a computer search we find the quartic polynomial 
$\bar{f}_q(x,y)=x^4 + y^3+ x^3 y + x y^2 + 1\in\mathbb{F}_q[x, y]$. 
Take $f_p(x,y)=x^4+y^4+x^2-y^2+3x$.
Then we obtain the plane quartic curve over $\mathbb{Q}$ with equation $f(x,y)=0$, where
\begin{equation*}
f(x,y)=x^4 + 486 x^3 y + y^4 + 486 x y^2 - 485 x^2 + 485 y^2 - 1455 x + 486.
\end{equation*}

\end{enumerate}
\end{exmp}

The rest of the section is devoted to the proof of Theorem \ref{thm:refined}.
For the convenience of the reader, we recall the contents of Theorem 3.10 from \cite{AAKRTV14}:
 Let $A$ be a principally polarised $n$-dimensional abelian variety defined over $\mathbb{Q}$.
Assume that  $A$ has semistable reduction of toric rank $1$ at some prime number $p$.
 Denote by $\Phi_p$ the group of connected components of 
 the N\'eron model of $A$ at $p$.
 Let $q$ be a prime of good reduction of $A$ and
 $P_q(X)=X^{2n} + aX^{2n-1} + \cdots + q^n\in \mathbb{Z}[X]$  the characteristic polynomial of the
 Frobenius endomorphism acting on the reduction of $A$ at $q$.
 Then for all primes $\ell$ which do not divide $6pqa\vert \Phi_p\vert $
 and  such that the reduction of $P_q(X)$ mod $\ell$ is irreducible
 in $\mathbb{F}_{\ell}$, the image of $\overline{\rho}_{A, \ell}$
 coincides with $\GSp_{2n}(\mathbb{F}_{\ell})$.

\begin{proof}[Proof of Theorem \ref{thm:refined}]
Fix a prime $\ell\geq 5$.
Let $q$ and $C_{q}$ be a prime, respectively a genus $3$ curve over $\mathbb{F}_{q}$, provided by Theorem~\ref{A}.
The curve $C_q$ is
either a plane quartic or a hyperelliptic curve. More precisely, it is defined by an equation $\bar{f}_q(x, y)=0$, where
$\bar{f}_q(x, y)\in \mathbb{F}_q[x, y]$ is a quartic type polynomial in the first case and a
$3$-hyperelliptic type polynomial otherwise (cf.~Subsection  \ref{subsec:notation}).
Note that if $f(x, y)\in\mathbb{Z}[x, y]$ is a
quartic (resp.~$3$-hyperelliptic type) polynomial which reduces to $\bar{f}_q(x, y)$ modulo $q$, then
it defines a smooth projective genus $3$ curve over $\mathbb{Q}$ which is geometrically irreducible.

Let now $p\not\in \{2, q, \ell\}$ be a prime.
Assume that $f(x, y)\in \mathbb{Z}[x, y]$ is a
polynomial of the same type as $\bar{f}_q(x,y)$ which is congruent to
$\bar{f}_q(x, y)$ modulo $q$ and also satisfies the two conditions of the statement of Theorem \ref{thm:refined} for this $p$.
We claim that the curve $C$ defined over $\mathbb{Q}$
by the equation $f(x, y)=0$ satisfies all the conditions of the explicit surjectivity result of
(\cite[Theorem 3.10]{AAKRTV14}).
Namely, Proposition \ref{prop:curve} implies that $C$ is a smooth projective geometrically
connected curve of genus $3$ with stable reduction. Moreover, according to Proposition~\ref{jacobianthm},
the Jacobian $\mathrm{Jac}(C)$ is a principally polarised $3$-dimensional abelian variety over
$\mathbb{Q}$, and its N\'eron model  has semistable reduction at $p$ with toric
rank equal to $1$. Furthermore, the component group $\Phi_p$ of the 
N\'eron model of $\mathrm{Jac}(C)$ at $p$ has order $2$.
Finally, by the choice of $q$ and $C_q$ provided by Theorem~\ref{A},  $q$ is a prime of good reduction of
$\mathrm{Jac}(C)$ such that the Frobenius endomorphism of the special fibre at $q$ has Weil polynomial $P_q(X)= X^6 + a X^5 + b X^4 + c X^3 + qbX^2 + q^2aX + q^3$, which is irreducible modulo $\ell$. Since the prime $\ell$ does not divide $6pqa\vert \Phi_p\vert $,
we conclude that the image of the Galois representation $\overline{\rho}_{\mathrm{Jac}(C), \ell}$ attached to the $\ell$-torsion
of $\mathrm{Jac}(C)$ coincides with $\mathrm{GSp}_6(\mathbb{F}_{\ell})$ by
Theorem~3.10 from \cite{AAKRTV14}.
\end{proof}

\section{Counting irreducible Weil polynomials of degree~$6$}\label{sectionirred}

In this section, we will prove Proposition~\ref{irredmodell} stated  in Section~\ref{sec:3}.
At the end of the section we present some examples.

This proof is based on Proposition~\ref{boundqweil} as well as Lemmas \ref{nnsquaresix} and \ref{redsix} below.

Let $\ell$ and $q$ be distinct prime numbers. 
Consider a polynomial of the form
\begin{equation}\label{weilpolyn} P_q(X)=X^6+a X^5+bX^4+cX^3+qbX^2+q^2aX+q^3 \in\Z[X].
 \tag{$\ast$}
\end{equation}

Proposition~\ref{boundqweil} ensures that for $q\gg\ell^2$, every polynomial \eqref{weilpolyn} with  coefficients in $]-\ell,\ell[$ is a Weil polynomial.
Then Lemmas~\ref{nnsquaresix} and \ref{redsix} allow us to show that the number of such polynomials which are irreducible modulo $\ell$  is strictly positive.

\begin{pr}\label{boundqweil}Let $\ell$ and $q$ be two prime numbers.
\begin{enumerate}
\item Suppose that $q>1.67\ell^2$. Then every polynomial $$X^4+uX^3+vX^2+uqX+q^2 \in \Z[X]$$ with integers $u,v$ of absolute value $<\ell$ is a Weil $q$-polynomial.
\item   Suppose that $q>1.82\ell^2$. Then every polynomial
\begin{equation*}
P_q(X)=X^6+a X^5+bX^4+cX^3+qbX^2+q^2aX+q^3 \in\Z[X],
\end{equation*}
with  integers $a,b,c$ of absolute value $<\ell$, is a Weil $q$-polynomial.
\end{enumerate}
\end{pr}

\begin{rmk} The power in $\ell$ is optimal, but the constants $1.67$ and $1.82$ are not. 
\end{rmk}

Let $D_6^{*-}$  be the number of polynomials of the form $P_q(X)=X^6+a X^5+bX^4+cX^3+qbX^2+q^2aX+q^3 \in\Z[X]$  with $a,b,c$ in  $[-(\ell-1)/2,(\ell-1)/2]$,  
$a,c\neq 0$ and  whose discriminant $\Delta_{P_q}$ is not a square modulo $\ell$, and $R_6$ 
the number of such polynomials which are Weil polynomials and are reducible modulo $\ell$.
Denoting by $\legendre{.}{\ell}$ the Legendre symbol, we have:
\begin{lm}\label{nnsquaresix} Let $\ell>3,$  then
$D_6^{*-}\geq \frac 12 (\ell-1)^2\left(\ell-1-\legendre q\ell\right) +\frac 12 (\ell-1)\legendre q\ell\left(1-\legendre{-1}\ell\right)-\ell(\ell-1).$
\end{lm}

\begin{lm}\label{redsix}
Let $\ell>3,$ then $R_6\leq \frac 38\ell^3 - \frac 5 8\ell^2\legendre q\ell - \ell^2 + \frac 32\ell\legendre q\ell + \frac 58\ell - \frac 38\legendre q\ell - \frac 12$.
\end{lm}

We postpone the proofs of Proposition~\ref{boundqweil} as well as Lemmas \ref{nnsquaresix} and \ref{redsix} 
to the following subsections but now use those statements to prove Proposition~\ref{irredmodell}.
Before that, let us recall a result of Stickelberger, as proven by Carlitz in \cite{Carlitz}, which will also 
be useful for proving Lemmas~\ref{nnsquaresix} and~\ref{redsix}: For  any monic polynomial $P(X)$ of degree $n$ 
with coefficients in $\Z$, and   any odd prime number $\ell$ not dividing its discriminant $\Delta_P$,  the number 
$s$ of irreducible factors of $P(X)$ modulo $\ell$ satisfies
\begin{equation}\label{eq:stickelberger}
\left( \frac{\Delta_P}{\ell} \right) = (-1)^{n-s}.
\end{equation}

\begin{proof}[Proof of Proposition~\ref{irredmodell}]
Let $\ell>3$ be a prime number. It follows from Stickelberger's result that  if $P_q(X)$ as in \eqref{weilpolyn} is  irreducible modulo $\ell$, 
then $\left( \frac{\Delta_{P_q}}{\ell} \right) = -1.$ Hence by Proposition~\ref{boundqweil}, when $q>1.82\ell^2$, we find that $(D^{*-}_6-R_6)$ is 
exactly the number of  degree $6$ ordinary Weil polynomials which have non-zero trace modulo $\ell$ and are irreducible modulo $\ell$.

By Lemmas~\ref{nnsquaresix} and \ref{redsix}, we have
$$ D_6^{*-}-R_6\geq \frac 18 \ell^3 + \frac 18 \ell^2\legendre{q}\ell - \frac 12\ell\legendre{-q}\ell - \frac 32\ell^2 + \frac 12\legendre{-q}\ell+ \frac{15}8\ell - \frac 58\legendre q\ell,$$ which is strictly positive  for all $q$, provided that $\ell\geq 13.$

For $\ell=3,5,7$ or $11$, direct computations of $(D_6^{*-}-R_6)$ using \textsc{Sage} show that $q=19$ for $\ell=3$, $q= 47$ for  $\ell = 5,\; q=97 $ for $\ell=7, \; q=223$ for $\ell=11$ will answer to the conditions of Proposition~\ref{irredmodell}.  Actually, computations indicate that for $\ell=5,7,11,$  $(D_6^{*-}-R_6)$  should be strictly positive  for any prime number $q$ and  for $\ell=3$,  it should be strictly positive  for  all prime numbers $ q$  which are not squares modulo $\ell$  (see  Remark~\ref{remirred}).
\end{proof}

\subsection{Proof of Proposition~\ref{boundqweil}}

Recall that $\ell$ and $q$ are two prime numbers.

We first consider degree $4$ polynomials. One can prove that a polynomial $X^4+uX^3+vX^2+uqX+q^2 \in \Z[X]$ is a $q$-Weil polynomial  if and only if the integers $u,v$ satisfy the following inequalities:

\begin{enumerate}[(1)]
\item\label{ineq1} $|u|\leq 4\sqrt{q}$, 
\item\label{ineq2} $2|u|\sqrt q-2q\leq v\leq \frac{u^2}4+2q$.

\end{enumerate}

\medskip

Let $q>1.67\ell^2 $ and $Q(X)=X^4+uX^3+vX^2+uqX+q^2 \in \Z[X]$ with $|u|<\ell,|v|<\ell$. Then $q \geq \frac 1{16}\ell^2$  and, 
since  $\ell\geq 2$, we have $q\geq \frac 14\ell^2\geq \frac12\ell$ so \eqref{ineq1} and  the right hand side inequality in \eqref{ineq2} are  satisfied. 
Finally, $q\geq \left(1+\frac{1}{2\sqrt 3}\right)^2\ell^2$ so $\sqrt q\geq \left(1+\frac 1{2\sqrt q}\right)\ell$ 
and the left hand side inequality in \eqref{ineq2} is satisfied. This proves that $Q(X)$ is a Weil polynomial and the first part of the proposition.

\medskip
Now we turn to degree $6$ polynomials. The proof is similar to the degree $4$ case.
According to Haloui \cite[Theorem~1.1]{Haloui10}, a  degree $6$ polynomial of the form \eqref{weilpolyn}
is a Weil polynomial if its coefficients satisfy the following inequalities:
\begin{enumerate}[(1)]
\item\label{condhaloui1} $|a|<6\sqrt q$, 
\item\label{condhaloui2} $4\sqrt q |a|-9q<b\leq \frac{a^2}3+3q$,
\item\label{condhaloui3} $-\frac{2a^3}{27}+\frac{ab}3+qa-\frac{2}{27}(a^2-3b^2+9q)^{\frac 32} \leq c \leq -\frac{2a^3}{27}+\frac{ab}3+qa + \frac{2}{27}(a^2-3b^2+9q)^{\frac 32} $,
\item\label{condhaloui4} $-2qa-2\sqrt q b-2q\sqrt q <c<-2qa+2\sqrt qb+2q\sqrt q$.
\end{enumerate}

\medskip
Let $q>1.82\ell^2$ and $P_q(X)$ a polynomial of the form \eqref{weilpolyn} with $|a|,|b|,|c|<\ell$. Then we note:
 \begin{itemize}

\item We have $q>\frac{1}{36}\ell^2$, so $\ell < 6\sqrt q$ and  \eqref{condhaloui1} is satisfied.
\item
The right hand side inequality of \eqref{condhaloui2} is satisfied since $\ell\leq 3q$.
 Moreover we have
$q>(1+\sqrt{17/8})\ell^2 \geq 4\ell^2(1+\sqrt{1+9/4\ell})^2/81$. Hence $9q-4\ell\sqrt q-\ell>0$ and the left hand inequality of  \eqref{condhaloui2} is satisfied.

\item A sufficient condition to have both inequalities in  \eqref{condhaloui3} is
 $$
2\ell^3+9\ell^2+27q\ell-2(-3\ell^2+9q)^{3/2}+27\ell\leq 0. $$
A computation shows that this inequality is equivalent to  $A\leq B$, with
\begin{align*}
A=\ell^6\left(\frac{28}{729}+\frac{1}{81\ell}+\frac {7}{108\ell^2}+\frac  1{6\ell^3}+\frac1{4\ell^4}\right) \mbox{ and } 
B= q^3\left(1-\frac54\frac{\ell^2}{q}+\frac{\ell^4}{q^2}\left(\frac 8{27}-\frac 1{6\ell}-\frac{1}{2\ell^2}\right)\right).
\end{align*}
 Since $\ell\geq 2$, we have $A\leq \frac{4537}{46656} \ell^6$
 and  $B\geq q^3\left(1-\frac 54\frac{\ell^2}{q}+\frac{19}{216}\frac{\ell^4}{q^2}\right)$.
Furthermore, since the polynomial
$$\frac{4537}{46656}X^3-\frac{19}{216}X^2+\frac54 X-1$$
has only one real root with approximate value $0.805$, we find that $A\leq B$,  because  $q\geq 1.243 \ell^2.$

\item Since $q>1.82\ell^2$ and $\ell\geq 2$, we have
$
\ell \left(\frac{1}{2q}+\frac{1}{\sqrt q}+1\right)\leq \ell\left(\frac 1{22}+\frac 1{\sqrt{11}}+1\right) <\sqrt q   $. 
Hence, $-2q\ell - 2\sqrt q \ell + 2q\sqrt q-\ell >0  $ and \eqref{condhaloui4} is satisfied.
 \end{itemize}
This proves that $P_q(X)$ is a Weil polynomial and the second part of the proposition. \hfill$\qed$

\subsection{Proofs of  Lemmas~\ref{nnsquaresix} and \ref{redsix}}

In this section, $\ell>2$, $q\neq \ell$ are prime numbers and we, somewhat abusively, denote with the same letter an integer in  $[-(\ell-1)/2,(\ell-1)/2]$ and its image in $\F_\ell$.

We will repeatedly use the following elementary lemma.

\begin{lm}\label{prelimun} Let $D\in\F_\ell^*$ and  $\varepsilon\in\{-1,1\}.$ We have
$$
\sharp \left\{x\in\F_\ell ;\legendre{x^2-D}{\ell} =\varepsilon\right\}=\frac 12\left(\ell-1-\varepsilon -\legendre D\ell\right);
$$
and $$\sharp\left\{(x,y)\in\F_\ell^2; \legendre{x^2-Dy^2}\ell=\varepsilon\right\}=\frac{1}{2}(\ell-1)\left(\ell-\legendre D\ell\right).$$
\end{lm}

\subsubsection{Estimates on the number of degree 4 Weil polynomials modulo $\ell$}

\begin{pr}\label{weilfour}
\begin{enumerate}
\item For $\varepsilon\in\{-1,1\}$, we denote by $D_4^\varepsilon$ the number of degree $4$ polynomials of the form 
$X^4+uX^3+vX^2+uqX+q^2\in\F_\ell[X]$ with discriminant $\Delta$ such that $\legendre \Delta \ell=\varepsilon$. Then
$$ D_4^-=\frac 12(\ell-1)\left(\ell-\legendre q\ell\right)\quad \mbox { and  } \quad
D_4^+=\frac 12(\ell-3)\left(\ell-\legendre q\ell\right)+1.$$
\item The number $N_4$ of degree $4$ Weil polynomials with coefficients in $[-(\ell-1)/2,(\ell-1)/2]$ which are irreducible modulo $\ell$ satisfies
\begin{equation}\label{nfour}
 N_4\leq\frac 14 (\ell+1)(\ell-1).
\end{equation}

\item The number $T_4$ of degree $4$ Weil polynomials with coefficients in $[-(\ell-1)/2,(\ell-1)/2]$  with exactly two irreducible factors modulo $\ell$ satisfies
\begin{equation}\label{tfour}
T_4\leq\frac 14(\ell-3)\left(\ell-\legendre q\ell\right)+\frac 18(\ell-1)(\ell+1).
\end{equation}
\end{enumerate}
\noindent Moreover, if $q>1.67\ell^2$, Inequalities (\ref{nfour}) and (\ref{tfour}) are equalities.
\end{pr}

\begin{proof}

\begin{enumerate}
\item First, we compute $D_4^{\varepsilon}$.
  The polynomial  $Q(X)=X^4+uX^3+vX^2+uqX+q^2$ has discriminant
$$\Delta = q^2 \kappa ^2 \delta
\quad \mbox{ where}\quad \kappa= -u^2 +4(v-2q) \quad\mbox{ and }\quad\delta =(v+2q)^2-4qu^2.$$

Since $q\in \F_\ell^*$,  we have $\legendre{\Delta}{\ell}=\legendre\kappa\ell^2\legendre\delta\ell$.
Moreover, notice that if $\kappa=0$ then  $\delta=(v-6q)^2$.

It follows that
$$D_4^- =\sharp\left\{(u,v)\in\F_\ell^2; \legendre{\delta}{\ell}=-1\right\}$$
and
$$D_4^+ =\sharp\left\{(u,v)\in\F_\ell^2; \legendre{\delta}{\ell}=1\right\}-\sharp\left\{(u,v)\in\F_\ell^2; v\neq 6q \mbox{ and } u^2=4(v-2q)\right\}.$$

Since the map $(u,v)\mapsto (v+2q,2u)$ is a bijection on $\F_\ell^2$ (because $\ell\neq 2$), by Lemma~\ref{prelimun} we have
$$\sharp\left\{(u,v)\in\F_\ell^2; \legendre{\delta}\ell =\varepsilon\right\}=
\sharp\left\{(x,y)\in\F_\ell^2; \legendre{x^2-qy^2}\ell=\varepsilon\right\}=\frac{(\ell-1)}2\left(\ell-\legendre q\ell\right)$$
for any $\varepsilon\in\{\pm 1\}.$ This gives the result for $D_4^-$. The result for $D_4^+$ follows from: 
\begin{eqnarray*}
\sharp\left\{(u,v); \; v\neq 6q \mbox{ and } u^2=4(v-2q)\right\}&=&
\sharp\left\{(u,v);\;  u^2=4(v-2q)\right\}-\sharp\{u\in\F_\ell; u^2=16q\}\\
&=&\ell-1-\legendre q\ell.
\end{eqnarray*}

\item Next, we bound the quantity $N_4$.   
By Stickelberger's result (see \eqref{eq:stickelberger}), 
 a  monic degree $4$ polynomial in $\Z[X]$ has non-square discriminant modulo $\ell$ if and only if  it has one or three distinct 
irreducible factors in $\F_\ell[X]$.   
In the latter case, 
the polynomial has the form $$(X-\alpha')(X-q/\alpha')(X^2-B'X+q)$$ with $X^2-B'X+q$ irreducible in $\F_\ell[X]$ and $\alpha'\neq q/\alpha'$ in $\F_\ell^*$. By Lemma~\ref{prelimun}, there are
$$\frac 14\left(\ell-2-\legendre{q}{\ell}\right) \left(\ell-\legendre q\ell\right)$$
such polynomials with three irreducible factors.
It follows that
\[
N_4\leq D_4^- - \frac 14\left(\ell-2-\legendre{q}{\ell}\right) \left(\ell-\legendre q\ell\right)\leq\frac 14(\ell-1)(\ell+1).
\]

\item Finally, we bound the quantity $T_4$. As 
above, Stickelberger's result implies that a degree $4$ Weil polynomial $Q(X)$ in $\Z[X]$ has exactly 
two distinct irreducible factors modulo $\ell$ if and only if $\legendre {\Delta_Q}{\ell}=1$  and  $Q(X) \pmod \ell$ does not have four distinct roots in $\F_\ell$.
By Lemma~\ref{prelimun}, there are
$$\frac 18\left(\ell-\legendre q\ell-2\right)\left(\ell-\legendre q\ell-4\right)$$
Weil polynomials  with coefficients in $[-(\ell-1)/2,(\ell-1)/2]$ whose reduction modulo $\ell$ has  
four distinct roots in $\F_\ell$.
It follows that
\begin{eqnarray*}
T_4&\leq&D_4^+ - \frac 18 \left(\ell-\legendre q\ell-2\right)\left(\ell-\legendre q\ell-4\right)\\
&\leq& \frac 14(\ell-3)\left(\ell-\legendre q\ell\right)+\frac 18(\ell-1)(\ell+1).
\end{eqnarray*}
\end{enumerate}
When $q>1.67\ell^2$, these upper bounds for $N_4$ and $T_4$ are equalities, since in this case, by Proposition~\ref{boundqweil},  every polynomial of the form $X^4+uX^3+vX^2+uqX+q^2$ with $|u|,|v|<\ell$  is a Weil polynomial.
\end{proof}

\subsubsection{Proof of Lemma~\ref{redsix}}

Recall that $R_6$ denotes the number of Weil polynomials $P_q(X)=X^6+aX^5+bX^4+cX^3+qbX^2+q^2aX+q^3$ 
with coefficients in $[-(\ell-1)/2,(\ell-1)/2]$, $a,c\neq 0$,  non-square discriminant modulo $\ell$ and which are reducible modulo $\ell$. We may drop  the  conditions $a\neq 0, c\neq 0$ to bound $R_6$. 

By Stickelberger's result (see \eqref{eq:stickelberger}), a monic degree $6$ polynomial in $\Z[X]$ with non-square discriminant modulo $\ell$  
has $1,\ 3$ or $5$ distinct irreducible factors in $\F_\ell[X]$. 
Hence, the factorisation in $\F_{\ell}[X]$ of  a polynomial  $P_q(X)$  as above is of one of the following types (note that a root $\alpha$ of $P_q(X)$ 
in $\overline{\F}_\ell$ is in $\F_\ell$ if and only $q/\alpha$  is also in $\F_\ell$):
\begin{enumerate}
\item $P_q (X)\equiv (X-\alpha)(X-\frac q\alpha)(X-\beta)(X-\frac q\beta)(X^2-CX+q)$, with $C^2-4q$ non-square modulo $\ell$ and $\alpha\neq q/\alpha$, $\beta\neq q/\beta$ and $\{\alpha,q/\alpha\}\neq\{\beta,q/\beta\}$; equivalently $P_q(X)\equiv (X^2-AX+q)(X^2-BX+q)(X^2-CX+q)$ where the first two quadratic polynomials are distinct and both reducible and the third one is irreducible;\item $P_q (X)\equiv (X-\alpha)(X-\frac q\alpha) Q(X)$, where $\alpha\neq q/\alpha$ and the irreducible factor  $Q(X)$ is the reduction of a  degree $4$ Weil polynomial;
\item  $P_q(X)$ is the product of three distinct irreducible quadratic polynomials, i.e., $P_q(X) \equiv (X^2-CX+q)Q(X)$ where $X^2-CX+q$ is irreducible and $Q(X)$ is the reduction of a degree $4$ Weil polynomial which has two distinct  irreducible factors, both of which are distinct from $X^2-CX+q$.
\end{enumerate}

We will count the number of polynomials of each type.

\noindent\textbf{Type 1.}
By Lemma~\ref{prelimun}, there are $\frac 12\left(\ell-\legendre q\ell\right)$ 
irreducible quadratic polynomials $X^2-CX+q$.
Also by Lemma~\ref{prelimun}, there are $\frac 12\left(\ell-2-\legendre q\ell\right)$ 
choices for reducible $X^2-AX+q$ without 
a double root and then there are $\frac 12\left(\ell-2-\legendre q\ell\right)-1$ choices 
for reducible $X^2-BX+q$ without a double root and distinct from $X^2-AX+q$.
It follows that there are
$ \frac 1{16}\left(\ell-\legendre q\ell\right)\left(\ell-\legendre q\ell-2\right)\left(\ell-\legendre q\ell -4\right) $
such polynomials.

\noindent\textbf{Type 2.}
By Proposition~\ref{weilfour} and Lemma~\ref{prelimun}, the number of polynomials with decomposition of this type is
$$\frac 12\left(\ell-\legendre q\ell-2\right)N_4 \leq \frac 18 (\ell+1)(\ell-1)\left(\ell-\legendre q\ell-2\right).$$

\noindent\textbf{Type 3.} Proposition~\ref{weilfour} and Lemma~\ref{prelimun} imply that there are
$$
\leq \frac 12\left(\ell-\legendre q\ell\right) T_4  \leq \frac 18\left(\ell-\legendre q\ell\right)^2(\ell-3)+\frac 1{16}(\ell-1)(\ell+1)\left(\ell-\legendre q\ell\right)
$$
polynomials of this type. \footnote{The first inequality is due to the fact that we do not take into account that $X^2-CX+q$ has to be distinct from the  factors of $Q(X)$.}

Summing these three upper bounds yields the lemma. \hfill$\qed$

\subsubsection{Proof of Lemma~\ref{nnsquaresix}}

The discriminant   of $P_q(X)$  is
$\Delta_{P_q}= q^6\Gamma^2\delta $,
where
 $$\Gamma= 8q a^4 + 9q^2a^2 - 42qa^2 b + a^2b^2 - 4a^3c + 108q^3 -
108q^2b + 36qb^2 - 4b^3 + 54qac + 18abc - 27c^2$$
and

$\delta=(c+2aq)^2-4q(b+q)^2.$  Hence, we have
 \begin{eqnarray*}
 D_6^{*-}&=&\sharp\left\{(a,b,c); a,c\neq 0, \Gamma\not\equiv 0\bmod \ell \mbox{ and } \legendre \delta\ell=-1\right\}\\
 &=&\sharp\left\{(a,b,c); a,c\neq 0, \legendre \delta\ell=-1\right\}
-\sharp\left\{(a,b,c); a,c\neq0, \Gamma\equiv 0\bmod \ell \mbox{ and } \legendre \delta\ell=-1\right\}\\
&\geq&M-W,
 \end{eqnarray*}
 where
$ M = \sharp\left\{(a,b,c); a,c\neq 0, \legendre \delta\ell=-1\right\}$ and $W =  \sharp\left\{(a,b,c); a\neq0, \Gamma\equiv 0\bmod \ell \right\}$.

 \paragraph{Computation of $M$.}
 Since $\ell >2$ and $q\in\F_\ell^*$, for any fixed $c\in\F_\ell^\times$, the map $(a,b)\mapsto (c+2aq,b+q)$ is a bijection from $\F_\ell^*\times\F_\ell$ to $\F_\ell\backslash\{c\}\times\F_\ell$.
 From this and Lemma~\ref{prelimun} we deduce that
 \begin{eqnarray*}
 M&=& \sum_{c\in\F_\ell^*} \sharp\left\{(x,y)\in\F_\ell^2; x\neq c,\; \legendre{x^2-4qy^2}{\ell}=-1\right\}\\
 &=& \sum_{c\in\F_\ell^*} \sharp\left\{(x,y)\in\F_\ell^2; \legendre{x^2-4qy^2}{\ell}=-1\right\} - \sum_{c\in\F_\ell^*} \sharp\left\{y\in\F_\ell ; \legendre{c^2-4qy^2}{\ell}=-1\right\}\\
 &=& \frac 12(\ell-1)^2\left(\ell-\legendre q\ell\right) - \sum_{c\in\F_\ell^*} M'_c,
 \end{eqnarray*}
where
\begin{eqnarray*}
 M'_c &=& \sharp\left\{y\in\F_\ell ; \legendre{c^2-4qy^2}{\ell}=-1\right\}\\
 &=&\sharp\left\{y\in\F_\ell ; \legendre{y^2-(c^2/4q)}{\ell}=-\legendre{-q}\ell\right\} \\
 &=& \frac 12\left(\ell-1-\legendre q\ell+\legendre{-q}\ell\right), 
\end{eqnarray*}
the last equality following from Lemma~\ref{prelimun}. 
This gives 
$$
M 
= \frac 12 (\ell-1)^2\left(\ell-1-\legendre q\ell\right) +\frac 12 (\ell-1)\legendre q\ell\left(1-\legendre{-1}\ell\right).
$$

\paragraph{Computation of $W=\sharp\left\{(a,b,c)\in\F_\ell^3; a\neq0, \Gamma=0 \right\}.$}

The discriminant of $\Gamma$ viewed as a\hfil \newline quadratic polynomial\footnote{More precisely, we have 
 $ \Gamma=-27c^2+G_1 c+G_0,\ (G_0,G_1\in \F_\ell[a,b])$ with 
 $G_1(a,b)=-2a (2a^2 - 27q -9b) $ 
and $G_0(a,b)=8qa^4 + 9q^2a^2 - 42qa^2b + a^2b^2 + 108q^3 - 108q^2b
+ 36qb^2 - 4b^3.$}
 in $c$ is  $\gamma=16(a^2-3(b-3q))^3.$ 
It follows that
\begin{eqnarray*}
W&=& 2\cdot\sharp\left\{(a,b)\in\F_\ell^2; a\neq 0, \legendre \gamma\ell=1\right\}+
\sharp\left\{(a,b)\in\F_\ell^2; a\neq 0, \gamma=0 \right\}\\
&=& 2\cdot\sharp\left\{(a,b)\in\F_\ell^2; a\neq 0, \legendre {a^2-3(b-3q)}\ell=1\right\}+
\sharp\left\{(a,b)\in\F_\ell^2; a\neq 0, a^2=3(b-3q) \right\}.
\end{eqnarray*}
Moreover, since  $\ell>3$, the map $b\mapsto 3(b-3q)$ is a bijection on $\F_\ell$.  So we have
\begin{eqnarray*}
W &=& 2\cdot\sharp\left\{(x,y)\in\F_\ell^2; x\neq 0, \legendre {x^2-y}\ell=1\right\}+
\sharp\left\{(x,y)\in\F_\ell^2; x\neq 0, x^2=y \right\}\\
&=& 2\cdot\sum_{y\in\F_\ell}\sharp\left\{x\in\F_\ell; \legendre{x^2-y}\ell=1\right\}-2\cdot\sharp\left\{y\in\F_\ell; \legendre{-y}{\ell}=1\right\}+\sum_{y\in\F_\ell^*}\sharp\{x\in\F_\ell^*; x^2=y\} \\
&=& \sum_{y\in\F_\ell^*}\left(\ell-2-\legendre y\ell\right)+2(\ell-1) - (\ell-1) + (\ell-1),
\end{eqnarray*}
using Lemma~\ref{prelimun}  (the second term is the contribution of $y=0$).
This yields $W=\ell(\ell-1)$ and computing $M-W$ concludes the proof. \qed

\subsection{Examples}\label{ex:section5}

This section contains examples of Weil polynomials satisfying the conditions in Proposition \ref{irredmodell}. They were
obtained using \textsc{Sage}.
\begin{itemize}
\item $\ell=3$, $q=19$: $P_q(X)=X^6 + X^5 + X^3 + 361X + 6859$;
\item $\ell=5$, $q=47$: $P_q(X)=X^6 + X^5 + X^4 + X^3 + 47X^2 + 2209X + 103823$;
\item $\ell=7$, $q=97$: $P_q(X)=X^6 + X^5 + 3X^3 + 9409X + 912673$;
\item $\ell=11$, $q=223$: $P_q(X)=X^6 + X^5 + 5X^3 + 49729X + 11089567$;
\item $\ell=13$:
\begin{itemize}
\item[] $q=311$: $P_q(X)=X^6 + X^5 + 3X^3 + 96721X + 30080231$;
\item[] $q=313$: $P_q(X)=X^6 + X^5 + 4X^3 + 97969X + 30664297$;
\item[] $q=317$: $P_q(X)=X^6 + X^5 + X^3 + 100489X + 31855013$;
\item[] $q=331$: $P_q(X)=X^6 + X^5 + 3X^3 + 109561X + 36264691$.
\end{itemize}
\end{itemize}

\bibliographystyle{plain}
\bibliography{Bibliography_Project2}

\end{document}